\RequirePackage{fix-cm}
\let\accentvec\vec
\documentclass[twocolumn]{svjour3}          % twocolumn

\let\vec\accentvec
\smartqed  % flush right qed marks, e.g. at end of proof
\usepackage{graphicx}
\usepackage{subfig}
\graphicspath{{fig/}}
\usepackage{mathptmx}      % use Times fonts if available on your TeX system
 \journalname{Statistics and Computing}

\usepackage{comment}
\usepackage{algorithm}
\usepackage{algpseudocode}

\usepackage{amsmath}
\usepackage{amsfonts}
\usepackage{amssymb}
\usepackage[left]{lineno}
\usepackage{caption}

\newcommand{\tr}{\operatorname{tr}}
\newcommand{\I}{\mathcal{I}}

\begin{document}
%\large

\title{Computing log-likelihood and its derivatives for restricted maximum likelihood\thanks{The research
was partially sponsored by the Engineering and
Physical Sciences Research Council (EPSRC, industry mathematics
knowledge transfer project, IM1000852) and is supported by National Natural Science of China (NSFC)(No.11501044), partially supported by (NSFC No. 11571002, 11571047, 1161049, 11671051, 61672003)}
}
%\subtitle{on the 400th anniversary of Napier's
%discovery (1614) of logarithms }

\titlerunning{Computing log determinant and its derivatives }
% if too long for running head

\author{
Shengxin Zhu %\and Sue Welham \and Robin Thompson \and Simon Harding \and
%Andrew J. Wathen
%etc.
}


\institute{ \emph{Present address:} of Shengxin Zhu \at
Laboratory of Computational Physics, Institute of Applied Physics and
Computational Mathematics, \\ P.O.Box, 8009, Beijing 100088, P. R. China.
\\ \email{zhu\_shengxin@iapcm.ac.cn}.
\\(previous) Oxford Center for Collaborative and Applied
Mathematics \& Numerical Analysis Group, Mathematical Institute,
The University of Oxford. Andrew Wiles Building, Radcliffe
Observatory Quarter, Woodstock Road, Oxford, OX2 6GG UK.
        Tel: 1865 615147 \\
           Fax: +44-1865-273583 \\
\email{zhus@maths.ox.ac.uk.}           %  \\
}

\date{Received: Jan, 13, 2015 / Accepted: date}

\maketitle

\begin{abstract}
Recent large scale genome wide association analysis involves large scale linear mixed models.
Quantifying (co)-variance parameters in the mixed models with a restricted maximum likelihood method
results in a score function which is the first derivative of a log-likelihood. To obtain a statistically
efficient estimate of the variance
parameters, one needs to find the
root of the score function via the Newton method.  Most elements of the Jacobian matrix of the score involve a trace term of four parametric matrix-matrix multiplications. It is computationally prohibitively for large scale data sets.
By a serial matrix transforms and an averaged information splitting technique, an approximate Jacobian matrix can be obtained by splitting the average of the Jacobian matrix and its expected value. In the approximated Jacobian, its elements only involve Four matrix vector multiplications which can be efficiently evaluated by solving sparse linear systems with the multi-frontal factorization method.

\end{abstract}

%\linenumbers

\newcommand{\var}{\operatorname{var}}
%\section{Problem description and motivation}

\section{Introduction}
Recent advance in genome-wide association study involves large scale linear mixed models \cite{lip11,lis12,zhang10,zhou12}.
Quantifying random effects in term of (co-)variance parameters in the linear mixed model is receiving increasing attention\cite{T08}. Common
random effects are blocks in experiments or observational studies that
are replicated across space or time \cite{Fisher1935,QK02}. Other random
effects like variation among individuals, genotypes and species also appear
frequently. In fact, geneticists and evolutionary biologists have long began
to notice the importance of quantifying magnitude of variation among
genotypes and spices due to environmental factors \cite{Hend59,MH08},
Ecologist recently are interested in the importance of random
variation in space and time, or among individual in the study of
population dynamics \cite{Bolker2008,PS03}. Similar problems also arises in estimating parameter
in high dimensional Gaussian distribution \cite{Daniel12}, functional
data analysis \cite{Bart14}, model selection analysis \cite{MSW13,YMO14} and many other applications \cite{RK88}.

Quantifying such random effects and making a statistical inference requires estimates of the co-variance parameters in the underlying model. The estimates are
usually obtained by maximizing a log-likelihood function which
often involves nonlinearly log-determinant terms. The first derivative of the log-likelihood is often referred to as a \emph{score function}. To maximize the log-likelihood, one requires to find the zeros of the score functions according to the conceptually simple Newton Method.  However, the Jacobian matrix of score function is very complicated (see \cite[p.825, eq.8]{zhou12},\cite[p.26, eq 11]{meyer96} and a derivation bellow). A remedy is the Fisher's scoring algorithm which uses the expectation of the Jacobian matrix in stead of the Jacobian matrix \cite{Longford1987}. The expect value of the Jacobian matrix is much simper than the Jacobian matrix but still involves a trace term of four matrix-matrix product. Such a trace term is computationally prohibitive for large data sets like those in genome wide association analysis. Therefore, effective way to evaluate the log-determinant terms and their derivative is of great interest.

Derivative free \cite{gras87} methods have been studied. They require less computational time per iteration, but they converge slow and require more iterates, especially for large scale problems \cite{misz93}. Comparisons in \cite{misz94b} shows that the derivative approach requires less time for most cases. That is why recent large scale genome wide association applications  \cite{lip11,lis12,zhang10,zhou12} and robust software development prefer the derivative approach. In this paper, we focus on a brute-fore approach to evaluate the elements of the Jacobian matrix. Simplified formula are obtained by a serial of matrix transforms and an averaged information splitting technique \cite{GTC95,john95,M97} which splits the Jacobian matrix and its expect value into two parts. The main part keeps the essential information and enjoys a simpler formula, the expectation of the other part which involves a lot of computations is negligible random zero matrix\cite{Z16,ZGL16}. Most elements of the approximated Jacobian matrix only involve a quadratic form of the observation vector. The quadratic from can be evaluated by four matrix-vector multiplications. The matrix vector multiplications are reduced to solve linear systems with multiple right hand sides, which can be efficiently solved by the multi-frontal or super-nodal $LDL^T$ factorization. These techniques enable the derivative Newton method applicable for high-throughput biological data \cite{WZW13}.

This paper provides a derivation for the derivatives of the log-likelihood. Since these results scattered in a serial of publications which are difficult for a new reader to track. We shall provide a detailed proof for this formulas to make the paper more readable. Based on this self-consistent derivation and some observations. We present efficient approaches to evaluate the log-likelihood methods and its second derivatives.  The evaluation of the likelihood relies on a $LDL^T$ factorization of a sparse matrix $C$. The order of $C$ is often less than the number of observations.  The evaluation of second derivatives of the likelihood finally reduced to solve many linear systems with the same coefficient matrix $C$. This supplies an alternative way to evaluate the derivatives as described in \cite{meyer96} which is based on automatic differentiation of the Cholesky algorithm \cite{smith95}.

\begin{table*}[t!]
\centering
 \caption{Comparison between the observed, Fisher and averaged splitting information}
 \label{tab:splitting}
\begin{tabular}{llll}
  \hline
  index & Jacobian & Expected Jacobian & Averaged Jacobian Splitting \\
   $-\frac{\partial^2 \ell_R}{\partial \theta_i \partial \theta_j} $    & $ \I_O$ & $\I $ & $\I_A$  \\
  $(\sigma^2, \sigma^2)$ &
  $\frac{y^TPy}{\sigma^6}-\frac{n-\nu}{2\sigma^4}$ & $\frac{n-\nu}{2\sigma^4}$
  & $\frac{y^TPy}{2\sigma^6}$  \\
  $(\sigma^2,\kappa_i)$ & $\frac{y^TPH_iPy}{2\sigma^4}$&
  $\frac{\tr(PH_i)}{2\sigma^2}$ &$\frac{y^TPH_iPy}{2\sigma^4}$\\
  $(\kappa_i,\kappa_j)$ &
 $\frac{\tr(PH_{ij})- \tr(PH_iP H_j)}{2} +\frac{2y^TPH_iPH_j Py -y^T
    PH_{ij}Py}{2\sigma^2}, $ & $\frac{\tr(PH_iPH_j)}{2}$
  & $\frac{y^TPH_iPH_jPy}{2\sigma^2}$\\
& $P=H^{-1}-H^{-1}X(X^TH^{-1}X)^{-1}X^TH^{-1}$ & $H_i=\frac{\partial H_i}{\partial \kappa_i}$, &
$H_{ij}=\frac{\partial^2H}{\partial \kappa_i\partial \kappa_j}$ \\
  \hline
\end{tabular}
\end{table*}

\section{Preliminary}

The basic model we considered is the widely used Linear Mixed Model(LMM).
\begin{equation}
y=X\tau+Zu+e. \label{eq:LMM}
\end{equation}
In the model, $y\in \mathbb{R}^{    n\times 1}$ is avector of observable
measurements
 $\tau\in \mathbb{R}^{p\times 1}$ is a vector of fixed
effects, $X \in \mathbb{R}^{n\times p}$ is a \textit{design
matrix} which corresponds to the fixed effects, $u \in
\mathbb{R}^{b\times 1}$ is a vector of random effects, $Z \in
\mathbb{R}^{n\times b}$ is a design matrix which corresponds
observations to the appropriate combination of random effects.
$e\in \mathbb{R}^{n\times 1}$ is the vector of residual errors.
The linear mixed model is an extension to the linear model
\begin{equation}
y=X\tau +e.    \label{eq:LM}
\end{equation}
LMM allows additional random components, $u$, as
correlated error terms, the linear mixed model is also referred to as \textit{linear mixed-effects models}. The term(s) $u$ can be
added level by level, therefore it is also referred to as\textit{hierarchal models}. It brings a wider range of
\textit{variance structures and models} than the linear model in
\eqref{eq:LM} does. For instance, in most cases, we suppose that
the random effects, $u$, and the residual errors, $e$, are
multivariate normal distributions such that $E(u)=0$, $E(e)=0$,
$u\sim N(0, \sigma^2 G)$, $e\sim N(0, \sigma^2 R)$ and
\begin{equation}
\text{var}\left[\begin{array}{c}
u\\
e
\end{array}\right]=\sigma^{2}\left[\begin{array}{cc}
G(\gamma) & 0\\
0 & R(\phi)
\end{array}\right],
\end{equation}
where $G\in \mathbb{R}^{b\times b }$, $R \in \mathbb{R}^{n\times
n}$. We shall denote $\kappa=(\gamma; \phi)^T.$
Under these
assumptions, we have
\begin{align}
y \vert  u &\sim N(X\tau+Zu, \sigma^2 R), \\
y &\sim N(X\tau, \sigma^2(R+ZGZ^T)):= N(X\tau, V(\theta)),
\end{align}
where $\theta=(\sigma^2;\kappa)^T$.
When the co-variance matrices $G$ and $R$ are known, one can obtain the \emph{Best Linear Unbiased Estimators} (BLUEs), $\hat{\tau}$, for the fixed effects and the \emph{Best Linear Unbiased Prediction} (BLUP), $\tilde{u}$, for the random effects according to the maximum likelihood method, the Gauss-Markov-Aitiken least square \cite[\S 4.2]{Rao}. $\hat{\tau}$ and $\tilde{u}$ satisfy the following mixed model equation\cite{Hend59}
\begin{equation}
\begin{pmatrix}
X^TR^{-1}X  & X^T R^{-1}Z \\
Z^TR^{-1}X  & Z^TR^{-1}Z+G^{-1}
\end{pmatrix}\begin{pmatrix}
\hat{\tau}  \\ \tilde{u}
\end{pmatrix}
=\begin{pmatrix}
X^TR^{-1}y \\ Z^T R^{-1}y
\end{pmatrix}.
\label{eq:mme}
\end{equation}
For such a forward problem, confidence or uncertainty of the estimations of the fixed and random effects can be quantified in term of co-variance of the estimators and the predictors
\begin{equation}
\mathrm{var} \begin{pmatrix}
\hat{\tau} \\
\tilde{u} -u
\end{pmatrix}
 =\sigma^2 C^{-1},
\label{eq:pre}
\end{equation}
where $C$ is the coefficient matrix in the mixed model equation \eqref{eq:mme}.

In many other more realistic and interesting cases. The variance parameter $\theta$ is unknown and to be estimated. This paper focuses on these cases.
One of the commonly used methods to estimate these parameters is
the maximum likelihood principle. In this approach, one starts with the distribution of the random vector $y$.
The variance of $y$ in the linear mixed model \eqref{eq:LMM} is
\begin{equation}
V = \operatorname{var}(y) %=\mathrm{var}(e)+Z\mathrm{var}(u)Z^{T}
    =\sigma^{2}(R+ZGZ^{T}):=\sigma^2H(\kappa),
\end{equation}
and the likelihood function of $y$ is
\begin{multline}
L(\tau,\theta)=
\prod_{i=1}^n (2\pi)^{-\frac{n}{2}} |V(\theta)|^{-\frac{1}{2}} \times \\
\exp\left\{-\frac{1}{2}(y-X\tau)^TV(\theta)^{-1}(y-X\tau) \right\}.
\end{multline}
Since the logarithmic transformation is monotonic, it is equivalent to maximize
$\log L(\tau, \theta)$ instead of $L(\tau,\sigma^2)$.  The log-likelihood function is
\begin{multline}
\log L(\tau,\theta)=-\frac{1}{2}\left\{n\ln(2\pi)+\ln|V(\theta)| +\epsilon^TV(\theta)\epsilon \right\}
\end{multline}
where $\epsilon=y-X\tau$. A maximum likelihood estimates for the variance parameter $\theta$ is
$$
\hat{\theta}=\arg_{\theta}\max \log L(\tau,\theta).
$$
The maximum likelihood estimate£¬$\hat{\sigma}^2$£¬for the variance parameter is asymptotically approaching to the true value,  $\sigma^2$, however, the
bias is relative large for finite observations with relative many effective fixed effects. Precisely
$$
\mathrm{Bias}(\hat{\sigma}^2,\sigma^2)=\frac{\nu}{n} \sigma^2,
$$
where $\nu=\mathrm{rank}(X)$.

A remedy to remove or at least reduce such a bias is the Restricted Maximum Likelihood (REML) \cite{PT71}, which is also referred to as \textit{marginal maximum likelihood} method or \textit{REsidual
Maximum Likelihood} method. In the framework of REML, the observation $y$ is divided into two (orthogonal) components: one of the
component of $y$ contains all the (fitted) residual error information
in the linear mixed model \eqref{eq:LMM}. Employing the maximum likelihood on the two orthogonal components results in two smaller problems (compared with the ML estimation). The partition is constructed as
follows.
\newcommand{\rank}{\operatorname{rank}}
For any $X\in
\mathbb{R}^{n\times p}$, there exist a linear transformation
$L=[L_1, L_2]$, such that $L_1^TX=I_p$ and $L_2^TX=0$ (See \cite{Verbyla1990}
Theorem \ref{thm:L}). Use this transform, we obtain

\begin{equation}
L^Ty=\begin{pmatrix}
y_1 \\ y_2
\end{pmatrix}
\sim
N\left( \begin{pmatrix}
\tau \\0
\end{pmatrix}, \sigma^2
\begin{pmatrix}
L_1^THL_1 & L_1^THL_2 \\
L_2^THL_1 & L_2^THL_2
\end{pmatrix}
 \right).
 \label{eq:y1y2}
\end{equation}
The marginal distribution of $y_2$ is given as
\cite[p40, Thm
2.44]{all04}
$$y_2=L_2^Ty\sim N(0,\sigma^2L_2^THL_2).$$
The associated log-likelihood function corresponding to $y_2$ is
\begin{multline}
\ell_R= \ell_2=\ell(\sigma^2,\kappa)
=-\frac{1}{2}\{(n-\nu)\log(2\pi \sigma^2)   \\
+\log|L_2^TH(\kappa)L_2|+y^TL_2(L_2^TH(\kappa)L_2)^{-1}L_2^Ty)/\sigma^2 \}.
\label{eq:l2}
\end{multline}
The REML estimate for the variance parameter is
$$
\hat{\theta}^{\mathrm{REML}}=\arg_{\theta}\max\ell_R(\theta).
$$
Such an estimate removes redundant
freedoms which are used in estimating the fixed effects. The fixed effects are determined by maximizing
the log-likelihood function of $y_1$.

The first derivatives of a log-likelihood function is referred to as the \emph{score function}. The REML estimate,if exists,
is a zero of the \textit{score} function, which can be approximated iteratively via the
Newton-Raphson method in Algorithm \ref{alg:NR}.

\begin{algorithm}[!t]
\caption{Newton-Raphson method to solve {$S(\theta)=0$.}}
\begin{algorithmic}[1]
\State {Give an initial guess of $\theta_0$}
\For{ $k=0, 1, 2, \cdots$ until convergence }
 \State{Solve $J_S(\theta_k) \delta_k=-S(\theta_k)$,// $J_s$ is the Jacobian matrix}
 \State{$\theta_{k+1}=\theta_k+\delta_k$}
\EndFor
\end{algorithmic}
\label{alg:NR}
\end{algorithm}

\section{Scores and its derivatives for REML}

\label{sec:comput}
\subsection{Closed formula for the restricted log-likelihood}
The restricted log-likelihood given in \eqref{eq:l2} involves an intermediate matrix $L_2$.  We shall give a closed formula of $\ell_R$ which is only related to the design matrix $X, Z$ and variance matrix $R$ and $G$.
This form is equivalent to the widely cited form:
\begin{theorem}
The residual log-likelihood for the linear model in
\eqref{eq:l2} is equivalent to
\begin{align}
\ell_{R} =& -\frac{1}{2}\left\{ (n-\nu)\log ( \sigma^2)+\log |H|
+\log|X^TH^{-1}X| \right\}   \nonumber \\
 &-\frac{1}{2}y^TPy/\sigma^2 +\mathrm{const} .
\end{align}
where $ H=R+ZGZ^T$ and
\begin{equation}
P=H^{-1}-H^{-1}X(X^TH^{-1}X)^{-1}X^TH^{-1}.\label{eq:PHX}
\end{equation}

\end{theorem}
\begin{proof}
 First we notice that $P=L_2(L_2^THL_2)^{-1}L_2$ (see
Appendix \eqref{eq:L2HL2}), and (Appendix \eqref{eq:XHX})
 \begin{equation}
 (X^TH^{-1}X)^{-1}=L_1^THL_1-L_1^{T}HL_2(L_2^THL_2)^{-1}L_2^THL_1.
 \end{equation}
Then we use the identity
\begin{align*}
 & \left\vert \begin{pmatrix}
I_p & -(L_1^THL_2)(L_2^THL_2)^{-1} \\
0 &I_{n-p}
\end{pmatrix}
\begin{pmatrix}
L_1^THL_1 & L_1^THL_2 \\
L_2^THL_1 & L_2^THL_2
\end{pmatrix} \right\vert \\
& =\left\vert
\begin{pmatrix}
(X^TH^{-1}X)^{-1} & 0\\
L_2^THL_1 & L_2^THL_2
\end{pmatrix}\right\vert=\vert L^THL \vert,
\end{align*}
We have $|L^THL|=|H||L^TL|=|(X^TH^{-1}X)^{-1}||L_2^THL_2|$ and
\begin{equation}
\log|L^TL|+\log|H|=\log|L_2^THL_2^T|-\log|X^TH^{-1}X|.
\end{equation}
Note the construction of $L$ does not depend on $\sigma^2$ and
$\phi$, therefore $\log|L^TL|$ is a constant. \flushright \qed
\end{proof}

\subsection{The score functions for residual log-likelihood}

\begin{theorem}[\cite{GTC95}]
Let $X\in \mathbb{R}^{n\times p}$ be full rank in the linear mixed model \eqref{eq:LMM} and the restricted log-likelihood function be given as $\ell_R(\theta)$ in \eqref{eq:l2}.  The scores of the residual
log-likelihood $\ell_R$ are given by
\begin{align}
s(\sigma^2)&=\frac{\partial \ell_R}{\partial \sigma^2}
=-\frac{1}{2}\left\{
\frac{n-\nu}{\sigma^2}-\frac{y^TPy}{\sigma^4}\right \}, \\
s(\kappa_i)&=\frac{\partial \ell_R}{\partial
\kappa_i}=-\frac{1}{2}\left\{ \tr(P \dot{H}_i) -
\frac{1}{\sigma^2}y^TP \dot{H}_iP y  \right\},
\end{align}
where $\dot{H}_i=\frac{\partial{H}_i}{\partial \kappa_i}$.

\end{theorem}
\begin{proof}
 First, according to Theorem~\ref{thm:d}, We know that $P=L_2(L_2^THL_2)^{-1}L_2^T)$.
Consider the residual loglikelihood function in \eqref{eq:l2},
It is obvious for the score for $\sigma^2$.
\begin{equation}
s(\kappa_i)=\frac{\partial\log |L_2^THL_2|}{\partial \kappa_i}+
\frac{1}{\sigma^2}\frac{\partial (y^TPy)}{\partial \kappa_i}.
\label{eq:sk}
\end{equation}
 Using the fact on matrix derivatives of log determinant\cite[p.305, eq.8.6]{Har97}
 $$\frac{\partial \log|A|}{\partial
\kappa}=\tr(A^{-1}\frac{\partial A}{\partial \kappa_i})$$
and the property of the trace
operation $\tr(AB)=\tr(BA)$
\begin{align}
&\frac{\partial \log(\lvert L_2^THL_2\rvert)}{\partial \kappa_i} =
\tr\left((L_2HL_2)^{-1} \frac{\partial (L_2^THL_2)}{\partial
\kappa_i}\right)
\nonumber \\
&=\tr\left( \underbrace{L_2(L_2^THL_2)^{-1} L_2^T}_{=P}
\dot{H}_i\right)=\tr\left(P\dot{H}_i\right).
\end{align}
 One easy way to calculate the second term in \eqref{eq:sk} is
obtained by applying the result on
on matrix derivatives of the inverse of a matrix
\cite[p.307,eq.8.15]{Har97}
\begin{equation*}
\frac{\partial A^{-1}}{ \partial \kappa_i}
=-A^{-1}\frac{\partial A}{\partial \kappa_i} A^{-1}.
\end{equation*}
We have
\begin{align}
\frac{\partial P}{\partial \kappa_i} &=\frac{\partial (L_2(L_2^THL_2)^{-1}L_2^T)}{\partial
\kappa_i}=L_2\frac{\partial (L_2^THL_2)^{-1}}{\partial
\kappa_i}L_2^T  \nonumber \\
& =-L_2(L_2^THL_2)^{-1}\frac{\partial (L_2^THL_2)}{\partial
\kappa_i} (L_2^THL_2)^{-1}L_2^T \nonumber \\
& =-\underbrace{L_2(L_2^THL_2)^{-1}L_2^T}_{=P}\dot{H}_i
\underbrace{L_2(L_2^THL_2)^{-1}L_2^T}_{=P}
\nonumber \\& =-P\dot{H}_iP
=\dot{P}_i \label{eq:dpxh}.
\end{align}
\flushright \qed
\end{proof}

\subsection{Jacobian of the score}
The negative of the Hessian matrix of a log-likelihood function, or the negative Jacobian of the score function, is often refereed to as the
\textit{observed information matrix},
\begin{equation}
%\I([\sigma^2,\kappa], [\sigma^2, \kappa])
\I_o=  -
\begin{pmatrix}
 \frac{\partial \ell_R}{\partial \sigma^2\partial \sigma^2} &
 \frac{\partial \ell_R}{\partial \sigma^2 \partial \kappa_i} &
 \cdots &\frac{\partial \ell_R}{\partial \sigma^2 \partial \kappa_j}
 &
 \\
\frac{\partial \ell_R}{\partial \kappa_i\partial \sigma^2 }
&\frac{\partial \ell_R}{\partial \kappa_i\partial \kappa_i} &
\cdots &
\frac{\partial \ell_R}{\partial \kappa_i\partial \kappa_j} \\
\vdots

 & \vdots & \ddots & \vdots
 \\
\frac{\partial \ell_R}{\partial \kappa_j\partial \sigma^2}

 & \cdots & \frac{\partial \ell_R}{\partial \kappa_j\partial \kappa_i}
 & \frac{\partial \ell_R}{\partial \kappa_j\partial \kappa_j}
\end{pmatrix}.
\end{equation}
In term of the observed information matrix, line 3 in Algorithm \ref{alg:NR} reads
as $\I_{o} \delta_k=S(\theta_k).$
\begin{theorem}[\cite{GTC95}]
Elements of the observed information matrix for the residual
log-likelihood \eqref{eq:l2} are given by
\begin{align}
\I_o(\sigma^2,\sigma^2) &=
\frac{y^TPy}{\sigma^6}-\frac{n-p}{2\sigma^4}, \label{eq:ISS} \\
\I_o(\sigma^2,\kappa_i) &=
\frac{1}{2\sigma^4}y^TP\dot{H}_iPy , \label{eq:ISK}\\
\I_o(\kappa_i,\kappa_j) & = \frac{1}{2}\left\{\tr(P\dot{H}_{ij})-
\tr(P\dot{H}_iP \dot{H}_j)
 \right\}  \nonumber  \\
 & +\frac{1}{2\sigma^2}\left\{ 2y^TP\dot{H}_iP\dot{H}_j Py -y^T
 P\ddot{H}_{ij}Py\right\}.
\label{eq:IKK}
\end{align}
where $\dot{H}_i=\frac{\partial H}{\partial \kappa_i}$,
$\ddot{H}_{ij}=\frac{\partial^2 H}{\partial K_i \partial K_j}$.
 \end{theorem}
 \begin{proof}
 The formulas in \eqref{eq:ISS} and \eqref{eq:ISK} and are standard.
 The first term in \eqref{eq:IKK} follows by applying the result in \eqref{eq:dpxh},
 \begin{align*}
 \frac{\partial \tr(P\dot{H}_i)}{\partial
 \kappa_j}&=tr(P\ddot{H}_{ij})+\tr(\frac{\partial P}{\partial \kappa_j}\dot{H}_i) \quad (\frac{\partial P}{\partial \kappa_j}=-PH_jP)\\
 &=\tr(P\ddot{H}_{ij})-\tr(P\dot{H}_jP\dot{H}_i).
 \end{align*}
 The second term in \eqref{eq:IKK} follows because of
 the result in \eqref{eq:dpxh},
 we have
 \begin{equation}
 -\frac{\partial (P\dot{H}_iP)}{\partial
 \kappa_j}=P\dot{H}_jP\dot{H}_iP-P\ddot{H}_{ij}P+P\dot{H}_iP\dot{H}_jP.
 \end{equation}
 Further note that $\dot{H}_i$, $\dot{H}_j$ and $P$ are symmetric.
The second term in \eqref{eq:IKK} follows because of
 $$y^TP\dot{H}_iP\dot{H}_jPy=y^TP\dot{H}_jP\dot{H}_iPy. $$

  \flushright \qed
 \end{proof}

The elements \eqref{eq:IKK} in the observed information matrix, the negative Jacobian matrix of the score, involve the trace product of four matrices. It is computationally prohibitive for large data
set. Therefore it is necessary to approximate the Jacobian matrix for efficiency.

\subsection{The Fisher information approximation to the negative Jacobian}
The \textit{Fisher information matrix}, $\I$, is the
expect value of the observed information matrix, $ \I=E(\I_o).$
The Fisher information matrix has a simpler form than the
observed information matrix and provides essential information on the observations, and thus it
is a nature approximation to the observed information matrix.

\begin{theorem}[\cite{GTC95}]
Elements of the Fisher information matrix for the residual
log-likelihood function in \eqref{eq:l2} are given by
\begin{align}
\I(\sigma^2,\sigma^2) &=E(\I_o(\sigma^2,\sigma^2))=\frac{\tr(PH)}{2\sigma^4}=\frac{n-p}{2\sigma^4}, \label{eq:FISS}\\
\I(\sigma^2, \kappa_i)
&=E(\I_o(\sigma^2,\kappa_i))=\frac{1}{2\sigma^2}
\tr(P\dot{H}_i),\label{eq:FISK} \\
\I(\kappa_i, \kappa_j)
&=E(\I_o(\kappa_i,\kappa_j))=\frac{1}{2}\tr(P\dot{H}_iP\dot{H}_j)
.\label{eq:FIKK}
\end{align}
\end{theorem}
\begin{proof} The formulas can be found in \cite{GTC95}. Here we supply alternative proof.
First note that
\begin{align}
PX&=H^{-1}X-H^{-1}X(X^TH^{-1}X)^{-1}XH^{-1}X=0. \nonumber \\
PE(yy^T)&=P(\sigma^2H-X\tau (X\tau)^T)=\sigma^2PH.
\label{eq:PEyy}
\end{align}
Then
\begin{align}
E(y^TPy)&=E(\tr(Pyy^T))=\tr(PE(yy^T)) =\sigma^2\tr(PH) \label{eq:Eypy}\\
&=\sigma^2\tr(L_2(L_2^THL_2)^{-1}L_2^TH) \nonumber \\
&=\sigma^2\tr((L_2^THL_2)^{-1}L_2^THL_2) \nonumber \\
&=\sigma^2\rank(L_2)=(n-p)\sigma^2. \nonumber
\end{align}
where $\rank(L_2)=n-\rank(X)$ due to $L_2^TX=0$.
Therefore
\begin{equation}
E(\I_o(\sigma^2,\sigma^2))=\frac{E(y^TPy)}{\sigma^6}-\frac{n-\nu}{2\sigma^4}=\frac{n-\nu}{2\sigma^4}.
\end{equation}
Second, we  notice that $PHP=P$. Apply the procedure in
\eqref{eq:Eypy}, we have
\begin{align}
E(y^TP\dot{H}_iPy)&= \tr(P\dot{H}_iPE(yy^T))=\sigma^2\tr(P\dot{H}_iPH)\nonumber \\
&=\sigma^2\tr(PHP\dot{H}_i) =\sigma^2\tr(P\dot{H}_i),  \label{eq:yphpy}\\
E(y^TP\dot{H}_iP\dot{H}_jPy) &=
\sigma^2\tr(P\dot{H}_iP\dot{H}_jPH)=\sigma^2\tr(PHP\dot{H}_iP\dot{H}_j) \nonumber \\
& =\sigma^2\tr(P\dot{H}_iP\dot{H}_j), \label{eq:yphphpy}\\
E(y^TP\ddot{H}_{ij}Py)&=\sigma^2\tr(P\ddot{H}_{ij}PH)=\sigma^2\tr(P\ddot{H}_{ij}).
\label{eq:yphhpy}
\end{align}
Substitute \eqref{eq:yphpy} into \eqref{eq:ISK}, we obtain
\eqref{eq:FISK}.   Substitute \eqref{eq:yphphpy} and
\eqref{eq:yphhpy} to \eqref{eq:IKK}, we obtain \eqref{eq:FIKK}.
 \flushright \qed
\end{proof}

Using the Fishing information matrix as an approximate to the negative Jacobian result in the
famous Fisher-scoring algorithm \cite{Longford1987}, which is widely used in machine learning.
\begin{algorithm}[!t]
\caption{Fisher scoring algorithm to estimate the variance
parameters}
\begin{algorithmic}[1]
\State {Give an initial guess of $\theta_0$} \For{ $k=0, 1, 2,
\cdots$ until convergence }
 \State{Solve $\I(\theta_k) \delta_k=S(\theta_k)$}
 \State{$\theta_{k+1}=\theta_k+\delta_k$}
\EndFor
\end{algorithmic}
\label{alg:FS}
\end{algorithm}

\subsection{Averaged information approximation to the negative Jacobian}
To avoid the trace term in the information matrix, the authors in \cite{GTC95}
suggest to use an average of the information matrix and the observed
information matrix. The \textit{average information matrix} is
constructed as follows
\begin{align}
\I_A(\sigma^2,\sigma^2)&=\frac{1}{2\sigma^6}y^TPy;  \label{eq:ASS}\\
\I_A(\sigma^2,\kappa_i)&=\frac{1}{2\sigma^4}y^TP\dot{H}_iPy;\label{eq:ASK}\\
\I_A(\kappa_i,\kappa_j)&=\frac{1}{2\sigma^2}y^TP\dot{H}_iP\dot{H}_jPy;
\label{eq:AKK}
\end{align}
In the paper \cite{GTC95}, the authors give an explanation of that average
information matrix can be viewed the average of the observed
information matrix and the Fisher information. This is true when the matrix $H(\kappa)$ has a linear structure with the parameter $\kappa$, say, $H=\sum V_i \kappa_i$. Then $\ddot{H}_{ij}=0$.
However in general $\I_A$ is not the average of the observed and expected information but only a main part of it. The following theorem gives a more precise and concise
mathematical explanation \cite{ZGL16,Z16}.
\begin{theorem}
Let $\I_O$ and $\I$ be the observed information matrix and the Fisher information matrix for the residual
log-likelihood of the linear mixed model respectively, then the average of the observed and the Fisher information can be split as $\frac{\I_O+\I}{2}=\I_A+ \I_Z$, such that the expectation of $\I_A$ is the Fisher information matrix and $E(\I_{Z})=0$.  \label{thm:S}
\end{theorem}
\begin{proof}
Let the elements of $\I_A$ are defined as in \eqref{eq:ASS} to \eqref{eq:AKK}
then we have
\begin{align}
\I_Z(\sigma^2,\sigma^2)&=0, \\
\I_Z(\sigma^2,\kappa_i)&=\frac{tr(P\dot{H}_i)}{4\sigma^2}-\frac{y^TP\dot{H}_iPy}{4\sigma^4}, \\
\I_Z(\kappa_i,\kappa_j)& =\frac{\tr(PH_{ij})-y^TPH_{ij}Py/\sigma^2}{4}.
\end{align}
Apply the result in \eqref{eq:Eypy}, we have
\begin{equation}
E(\I_A(\sigma^2,\sigma_2))=\frac{(n-p)}{2\sigma^4}=\I(\sigma^2,\sigma^2).
\end{equation}
Apply the result in \eqref{eq:yphpy}, we have
\begin{equation}
E(\I_A(\sigma^2,\kappa_i))=\frac{\tr(P\dot{H}_i)}{2\sigma^2} \text{ and }E(\I_Z(\sigma^2,\kappa_i))=0.
\end{equation}
Apply the result in \eqref{eq:yphphpy}, we have
\begin{equation}
E(\I_A(\kappa_i,\kappa_j))=\frac{\tr(P\dot{H}_iP{H}_j)}{2}=\I(\kappa_i,\kappa_j)
\end{equation}
and $E(\I_Z(\kappa_i,\kappa_j))=0$.
\qed
\end{proof}
Similar to the Fisher information matrix, Theorem \ref{thm:S} indicates that
the averaged information splitting matrix is a good approximation to the observed
information and can be used as an alternative Fisher information matrix.
%The elements of the average information matrix only involve a
%quadratic form, which is much easier to be evaluated.
\begin{algorithm}[!t]
\caption{Average information(AI) algorithm  to solve
{$S(\theta)=0$.}}
\begin{algorithmic}[1]
\State {Give an initial guess of $\theta_0$} \For{ $k=0, 1, 2,
\cdots$ until convergence }
 \State{Solve $\I_A(\theta_k) \delta_k=S(\theta_k)$,}
 \State{$\theta_{k+1}=\theta_k+\delta_k$}
\EndFor
\end{algorithmic}
\label{alg:AI}
\end{algorithm}

\section{Computing issues}
\subsection{Computing elements of approximated Jacobian}

Compare $\I_A $ with $\I_O$, and $\I_F$ in Table \ref{tab:splitting}, in contrast with $\I_{O}(\kappa_i,\kappa_j)$ which involves 4 matrix-matrix products, The counter part of the negative of approximated Jacobian only involves a quadratic term, which can be evaluated by four matrix-vector multiplications and an inner product as in Algorithm~\ref{alg:IKK}.This provide a simple formula. Still the matrix vector multiplication of $Py$ involves the inverse of the $H$ which is of order $n\times n$. When the observations is greater than the number of fixed and random effects, say $n>p+b$, we can obtain a much simpler matrix vector multiplication as $R^{-1}e$, where $e$ is the fitted residual $e= y-X\hat{\tau} -Z\tilde{u}$.

\begin{algorithm}
\caption{Compute $\I_A(\kappa_i,\kappa_j)=\frac{y^TP\dot{H}_iP\dot{H}_jPy}{2\sigma_2}$}
\label{alg:IKK}
\begin{algorithmic}[1]
\State{ $\xi =Py$ }
\State{ $\eta_i =H_i \xi$; $\eta_j=H_j\xi$};
\State{ $\zeta= P\eta_j $}
\State{ $\I_A(\kappa_i ,\kappa_j)=\frac{\eta_i^T \xi}{2\sigma^2}$}
\end{algorithmic}
\end{algorithm}
 We introduce the following lemma.
\begin{lemma}[{\cite[Fact 2.16.21]{B09}}]
Let $H=R+ZGZ^T$, then
$$
H^{-1}  =R^{-1} -R^{-1}Z(Z^TR^{-1}Z+G^{-1})^{-1} Z^TR^{-1}.
$$
\label{lem:H-1}
\end{lemma}
\begin{lemma}
The inverse of the matrix $C$ in \eqref{eq:mme} is given by
\[
C^{-1} =
\begin{pmatrix}
A& B \\
B^T & D
\end{pmatrix}^{-1}=
\begin{pmatrix}
C^{XX}  & C^{XZ} \\
C^{ZX} & C^{ZZ} \\
\end{pmatrix}
\]
where
\begin{align}
C^{XX} & =(X^TH^{-1}X)^{-1},\\
C^{XZ} & =-C^{XX}X^TR^{-1}ZD^{-1}, \\
C^{ZX} &= -D^{-1} Z^TR^{-1}XC^{XX}, \\
C^{ZZ} & =D^{-1}+C_{ZZ}^{-1}Z^TR^{-1}XC^{XX}X^TR^{-1} Z^TD^{-1}.
\end{align}
\label{lem:C1}
\end{lemma}
\begin{proof}
According to Fact~\cite[Fact 2.17.3]{B09} on $2\times 2$ partitioned matrix, $C^{-1}$ is given by
$$
\begin{pmatrix}
S^{-1}   & -S^{-1}BD^{-1} \\
-D^{-1} B^TS^{-1} & D^{-1}B^TS^{-1}BD^{-1}+D^{-1}.
\end{pmatrix}
$$
where $S=A-BD^{-1}B^T$. So we only need to prove
\begin{align*}
C^{XX} & =((X^TR^{-1}X)^{-1}-(X^TR^{-1}Z) D^{-1}(Z^TR^{-1}X))^{-1} \\
 & = (X^T \underbrace{(R^{-1}-R^{-1}Z(Z^TR^{-1}Z+G^{-1})^{-1}Z^TR^{-1})}_{H^{-1}} X)^{-1} \\
 & =(X^TH^{-1}X)^{-1}.
\end{align*}
\qed
\end{proof}

We shall prove the following results
\begin{theorem}
Let $P$ be defined in \eqref{eq:PHX}, $\hat{\tau}$ and $\tilde{u}$ be the solution to \eqref{eq:mme}, and $e$ be the residual $e=y-X\hat{\tau}-Z\tilde{u}$, then
\begin{align}
P&=H^{-1} -H^{-1}X(X^TH^{-1}X)^{-1}X^TH^{-1} \\
 & =R^{-1} -R^{-1}WC^{-1}W^TR^{-1} \label{eq:P2}
\end{align}
where $W=[X,Z]$ is the design matrix for the fixed and random effects and
\[
 Py=R^{-1} e.
\]
\end{theorem}
\begin{proof} Suppose \eqref{eq:P2} hold, then
\begin{align}
Py &=R^{-1}y-R^{-1}W \underbrace{C^{-1}W^TR^{-1} y}_{(\hat{\tau}^T, \tilde{u}^T)^T} \\
  & =R^{-1}(y -X\hat{\tau} -Z\tilde{u}) =R^{-1}e
\end{align}
\begin{align*}
&  R^{-1} -R^{-1}WC^{-1}W^TR^{-1} \\
=& R^{-1}
-R^{-1}(X, Z) \begin{pmatrix} C^{XX} &  C^{XZ}  \\ C^{ZX} & C^{ZZ}
\end{pmatrix} \begin{pmatrix} X^T \\Z^T \end{pmatrix} R^{-1}. \\
= & R^{-1} -R^{-1}\{ XC^{XX}X^T -XC^{XZ}Z -ZC^{ZX} + ZD^{-1}Z \\
& + Z(C_ZZ^{-1}Z^TR^{-1}XC^{XX}X^TR^{-1} Z^TD^{-1})Z^T \}R^{-1} \\
 = &\underbrace{R^{-1} -R^{-1}ZD^{-1}Z^TR^{-1}}_{H^{-1}}  \\
 & -(R^{-1} -R^{-1}ZD^{-1}Z^TR^{-1}) XC^{XX}X^{T}H^{-1} \\
= & H^{-1} -H^{-1}X(X^TH^{-1}X)^{-1}X^TH^{-1}.
\end{align*} \qed
\end{proof}
From above results, we find out that evaluating the matrix vector $Py$ is equivalent the solve the linear system \eqref{eq:mme}
\begin{equation}
\begin{pmatrix}
X^TR^{-1}X  & X^T R^{-1}Z \\
Z^TR^{-1}X  & Z^TR^{-1}Z+G^{-1}
\end{pmatrix}\begin{pmatrix}
\hat{\tau}  \\ \tilde{u}
\end{pmatrix}
=\begin{pmatrix}
X^TR^{-1}y \\ Z^T R^{-1}y
\end{pmatrix}.
\end{equation}
and then evaluate the weighted residual $R^{-1}e$. Notice that the matrix $P \in \mathbb{R}^{n\times n}$. On contrast, $C \in \mathbb{R}^{(p+b)\times (p+b)}$ where $p+b$ is the number of fixed effects and random effects. This number $p+b$ is much smaller than the number of observations $n$. In each nonlinear iterations, the matrix $C$  can be pre-factorized for evaluating $\xi=Py$. And then use the factors to solve $\zeta_i=P\eta_i$ (by solving the mixed equations with multiple right hand sides).

Before discuss technique details for efficient factorization, we shall demonstrate a by-product, an efficient formula for evaluating the restricted log-likelihood function.

\subsection{Evaluating the log-likelihood}
 An observation is that $\log \lvert C\rvert $ is ready when a $LDL^T$ factorization is obtained. And the second observation is that in many cases, the covariance matrix $G$ for the random effects and $R$ for the residual enjoy diagonal or block diagonal structures. Therefore $\log \lvert R \rvert$ and $\log \lvert G \rvert$ is easy to obtain in these cases.

\begin{theorem}
Let $H=R+ZGZ^T$, $C$ be the coefficient matrix in the mixed model equation \eqref{eq:mme}, then
\[
\log \lvert H\rvert +\log \lvert XH^{-1}X^T \rvert = \log \lvert C \rvert +\log \lvert R \rvert +\log \lvert G\rvert.
\]
\end{theorem}
\begin{proof}
Consider block elimination of the matrix $C$
\[
\begin{pmatrix}
A & B \\
B^T & D
\end{pmatrix}
\begin{pmatrix}
I & 0 \\
-D^{-1}B^T & I
\end{pmatrix}
=
\begin{pmatrix}
S & A \\
0& D
\end{pmatrix}
\]
where $S= A-BD^{-1}B^T=X^TH^{-1}X $ according to Lemma \ref{lem:C1}, $D=G^{-1} +Z^TR^{-1}Z$.
Therefore we have
\begin{equation}
\log\lvert C\rvert = \log\lvert X^TH^{-1}X \rvert +\log \vert G^{-1}+Z^TR^{-1}Z \rvert. \label{eq:LR2}
\end{equation}
Then consider the block elimination of the following matrix
\[
\begin{pmatrix}
R & Z \\
-Z^T & G^{-1}
\end{pmatrix}
\begin{pmatrix}
I & 0 \\
GZ^T & I
\end{pmatrix}
=\begin{pmatrix}
R+ZGZ^T & R \\
0 & G^{-1}
\end{pmatrix}=\begin{pmatrix}
H & G  \\
0 & G^{-1}
\end{pmatrix}.
\]
Similarly, we have
\[
\begin{pmatrix}
I & 0 \\
Z^TR^{-1} & I
\end{pmatrix}
\begin{pmatrix}
R & Z \\
-Z^T & G^{-1}
\end{pmatrix}
=\begin{pmatrix}
R & Z\\
0 & G^{-1} + Z^TR^{-1}Z
\end{pmatrix}
\]
Therefore we have
\begin{equation}
\lvert R\rvert \lvert G^{-1}+Z^TR^{-1}Z \rvert =\lvert H\rvert \lvert G^{-1} \rvert. \label{eq:LR3}
\end{equation}
Notice that $\lvert R^{-1} \rvert = \lvert R \rvert^{-1}$, and combine \eqref{eq:LR2} and \eqref{eq:LR3} we can obtain the required result. \qed
\end{proof}

\subsection{Multi-frontal $LDL^T$ factorization}

Factorizing a symmetric positive definite matrix $C$ is
usually done by the Cholesky factorization. This
can be efficiently implemented by the $LDL^T$ factorization,
where $L$ is a unit lower triangular matrix with diagonal
elements 1, $D$ is a diagonal matrix. The classical Cholesky factorization
$C=\tilde{L} \tilde{L}^T$ can be obtained form the $LDL^T$ factors by
setting $\tilde{L}=LD^{1/2}$; in practice, such a transform is not
necessary and the matrix $D$ can be stored in the diagonal part of
the matrix $L$ to save memory.

The $LDL^T$ factorization algorithm
can be derived from the following well known formula which we shall use
to derive the inversion formula.
\begin{linenomath}
\begin{align}
C=\left(
\begin{array}{cc}
\alpha &  a^T \\
a & \hat{C} \\
\end{array}
\right)&=
 \left(\begin{array}{cc}
 1 &    \\
\ell & I \\
\end{array} \right)
\left(
\begin{array}{cc}
\alpha &  \\
 & S \\
\end{array}
\right)
\left(
\begin{array}{cc}
1 & \ell^T \\
 & I \\
\end{array}
\right)  \label{eq:ldl1}\\
&=\left(\begin{array}{cc}
 1 &    \\
\ell & \hat{L} \\
\end{array} \right)
\left(
\begin{array}{cc}
\alpha &  \\
 & \hat{D} \\
\end{array}
\right)
\left(
\begin{array}{cc}
1 & \ell^T \\
 & \hat{L}^T \\
\end{array}
\right)
\label{eq:ldl2},
\end{align}
\end{linenomath}
where $\ell = a/\alpha$, $S=\hat{L}\hat{D}\hat{L}^T=\hat{C}-a
a^T/\alpha=\hat{C}-\ell a^T$. $S$ is usually
refereed as the \emph{Schur complement}. By recursively using the formula
\eqref{eq:ldl1} $n-1$ times to these Schur complements generated in the
process, one can derive one algorithm for the $LDL^T$
factorization.

Suppose there are $m_1$ elements the first column of $L$, this
includes the first diagonal element of $D$ which can be saved in
the diagonal position of $L$, then computing $\ell$ in
\eqref{eq:ldl1} requires  $m_1-1$ floating point
operations (FLOPS); with consideration of the symmetry,
updating the Schur complement $S$ requires $m_1(m_1-1)$
FLOPs. Let $m_i$,
$i=1,2,\ldots,n$ be the number of
non-zero elements of in the
$i$-th column of $L$. Then computing the $L$ and $D$ factors for
the $LDL^T$ factorization requires
\begin{equation}
 \sum_{i=1}^{n} m_i^2 -n
\label{eq:cldl}
\end{equation}
floating point operations. Therefore the FLOPS required for the $LDL^T$ factorization depend on the sparsity of the factor $L$. The non-zero pattern of $L$ and column counts $m_i$ can be
analysed by \textit{elimination tree} \cite[p.56]{Tim06}\cite{liu90} through symbolic $LDL^T$ factorization. This can be used to analyse the data flow dependence of the numerical $LDL^T$ factorization.

The worst case for the formula \eqref{eq:ldl1}
and \eqref{eq:ldl2} arises when the first column of $C$ is
dense; in such a case, the out product $aa^T$ is
dense and it results in a dense Schur component $S$, thus
one can never get benefit of the sparsity of $C$. To avoid such cases, a \emph{fill-in} reducing algorithm is required.
Commonly used fill-in reducing methods are minimum degree ordering \cite{GL89},
nested dissection ordering \cite{G73}, and the approximate minimum degree(AMD)
ordering \cite{ADD96,ADD04}. An efficient implement of these techniques is available in the $LDL^T$ factorization package \cite{alg849}.

\section{Discussion}

\begin{table*}[!t]
\centering
\begin{tabular}{r|*{10}{r}}
\hline
Data \newline Set&y	& c & v & y.c	& y.v&  v.c	& units &
v/y &	 y/v	& c.v \\
\hline
P1&	12&	22&	130&	132&	673&	2518&
6667&
56.1&	5.2&	10 \\
P2&	15&	25&	160&	180&	888&	3527&
9595&
59.2&	5.6&	10 \\
P3&	22&	25&	188&	264&	1177&	4215&
12718&
53.5&	6.3&	12 \\
P4&	25&	25&	262&	300&	1612&	5907&
17420&
64.5&	6.2&	12\\
P5&	25&	25&	390&	300&	2345&	8625&
25334&
93.8&	6.0&	15 \\
P6&	25&	35&	390&	425&	2345&	12249&
35887&
93.8&	6.0&	15\\
P7&	30&	35&	470&	510&	3013&	15087&
46113&
100.4&	6.4&	20\\
P8&	30&	35&	620&	510&	3835&	19737&
58685&
127.8&	6.2&	20\\
P9&	35&	40&	720&	700&	4522&	26432&
81396&
129.2&	6.3&	20\\
P10&	40&	50&	820&	1000&	5262&	37701&
118403&	131.6&	6.4&	20 \\
\hline
\end{tabular}
\caption{Data sets for the benchmark problem. There column
titles are  the number of years (y), the number of centres (c), the
number of varieties(v), thes number of levels of cross terms(y.c
, y.v,  v.c),  the average varieties per year (v/year), and the
averages year per variety (y/v), and  the number of controlled varieties
all year (control varieties all year).}
\label{tab:bench1}
\end{table*}
The paper details that the elements of the Hessian matrix of the log-likelihood function can be computed by solving the mixed model equations.
Matrix transforms play an important role in splitting the average Jacobian matrices of the score function. Such a splitting results in a simple approximated Jacobian matrix which reduces computations
form four matrix-matrix multiplications to four matrix-vector multiplications. This significantly reduces the time for evaluating the Jacobian matrix in the Newton method.  The problem of evaluating the Jacobian matrix of the score function finally is reduced to solving the mixed model equations \eqref{eq:mme} with multiple right hand sides. At the end of the day, an efficient sparse factorization method plays a crucial role in evaluation the Hessian matrix of the log-likelihood function.

Finally, we demonstrate an variety trial problem to illustrated the current performance in of the factorization method for the mixed model equation.
These examples are based on a second-stage analysis of a set of
variety trials, i.e. based on variety predicted values from each
trial. Trials are conducted in a number of years across a number
of locations (centres). See Table \ref{tab:bench1}.
They are sampled at
random, and the life of each variety is generated from a Poisson
distribution. This gives a three-way crossed structure
(year*variety*site) with some imbalance. In the current model, all
terms except a grand mean are fitted as random. The random terms
are generated as independent and identically distributed normal
distribution with variance components generated
from a test program.
\begin{figure}[!h]
\centering
 \subfloat[ \label{Csub1}] {\includegraphics[width=0.24\textwidth]{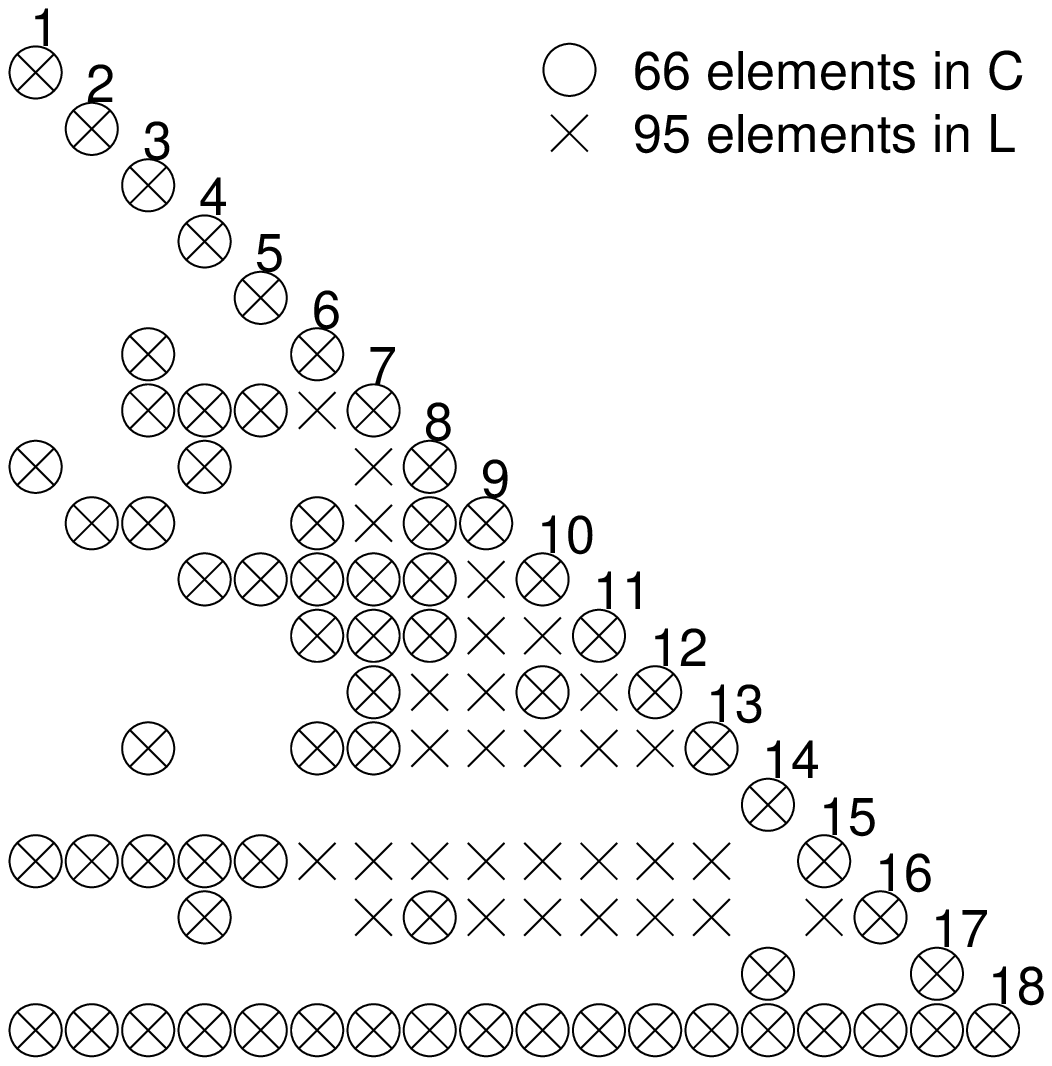}}
 \subfloat[\label{Csub1btree}] {\includegraphics[width=0.24\textwidth]{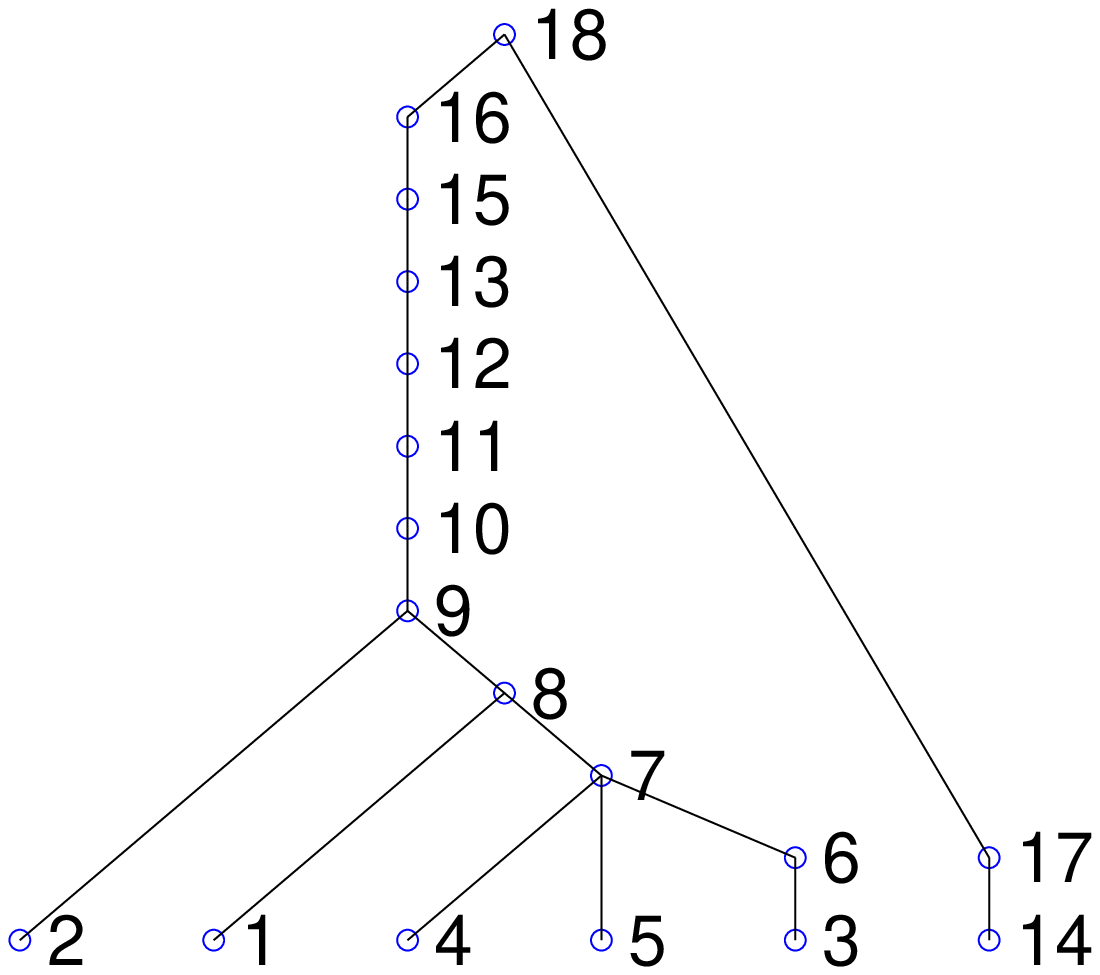}}

 \subfloat[ \label{Csub2}] {\includegraphics[width=0.24\textwidth,]{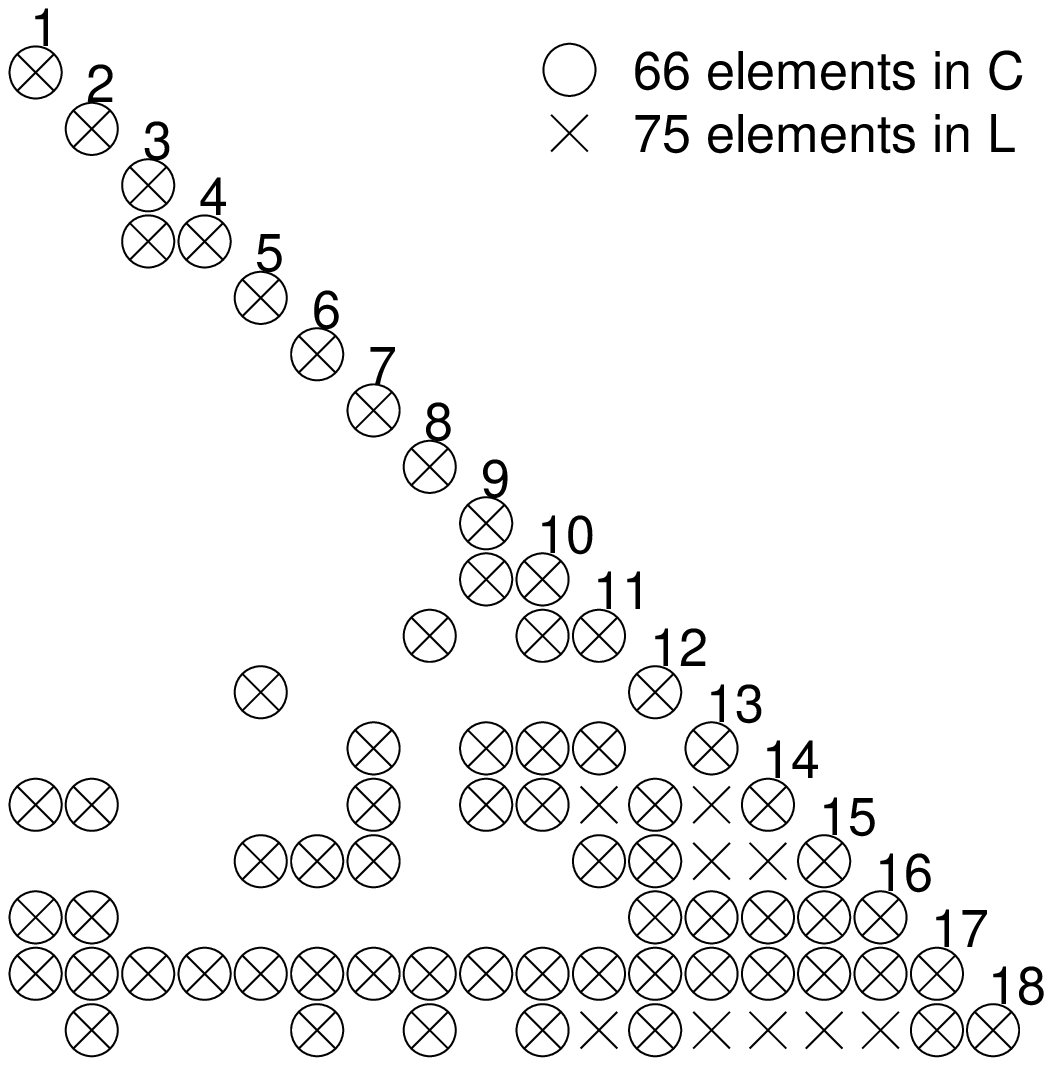}}
 \subfloat[\label{Csub2tree} ] {\includegraphics[width=0.24\textwidth,]{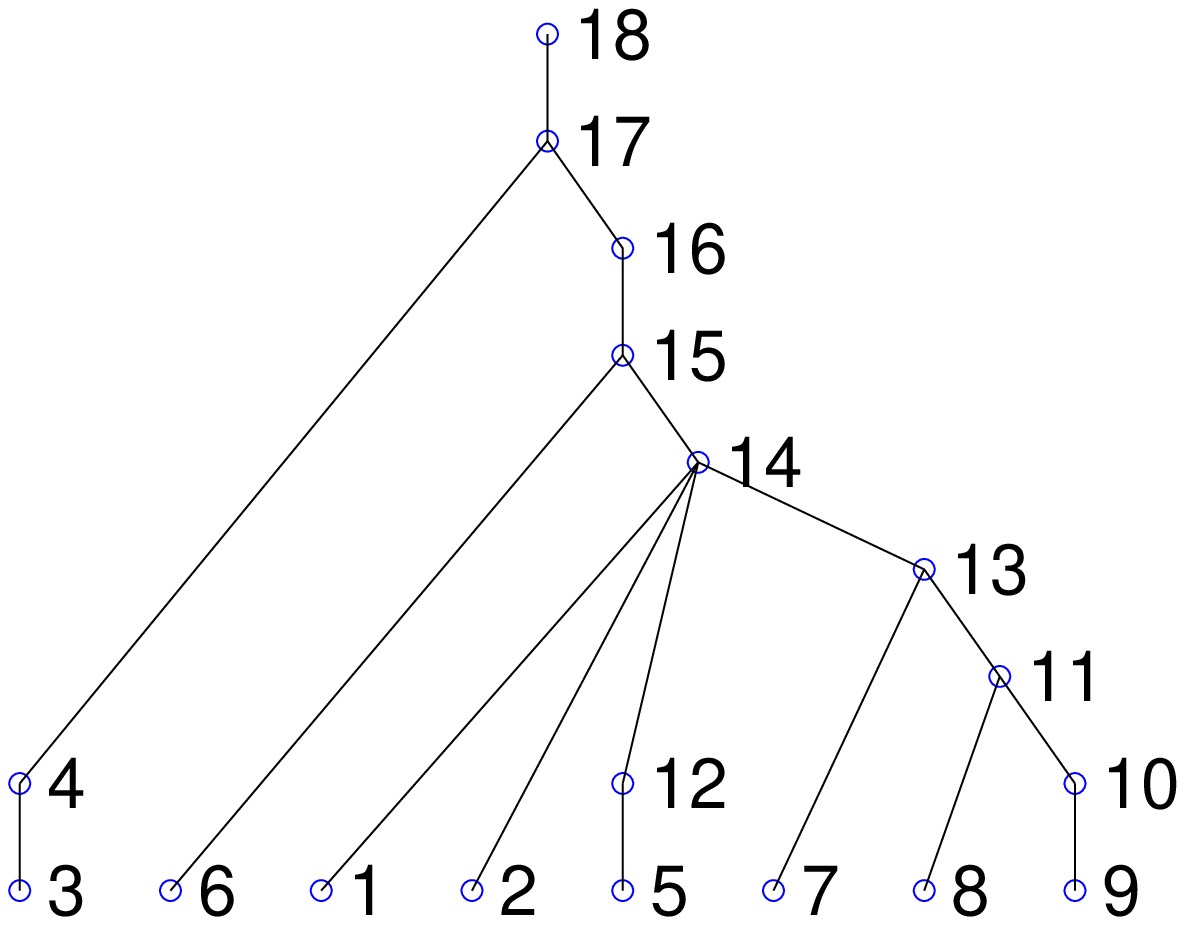}}

\caption{The fill-in reducing ordering and its effects.
Window \protect\subref{Csub1} demonstrates the sparse pattern of  a matrix
$C$ and its $LDL^T$ factors $L$. The matrix $C$ is a sub matrix
extract from a bench mark problem of the linear mixed model.
Window \protect\subref{Csub1btree} shows the elimination tree for the
$LDL^T$ factorization of the matrix in \protect\subref{Csub1}; the
factorization starts from leaf nodes (column 1, 2, 3, 4, 5 and 14
in the matrix in \protect\subref{Csub1}) and finishes at the root node
(column 18 in the matrix in \protect\subref{Csub1}). The factorization of
a parent column depends on the data in its children columns. The
height of the tree shows the sequential steps in the
factorization. The width of the tree shows the possible maximum
parallelism.  Window \protect\subref{Csub2} illustrates the approximate
minimum degree ordering of the matrix in Window \protect\subref{Csub1}
and the $LDL^T$ factor corresponding the the AMD ordering.  After
the AMD ordering, there are only 9 fill-ins while the are 29
fill-ins in \protect\subref{Csub1}. Windows \protect\subref{Csub2tree}
demonstrates the elimination tree for the the AMD ordering, the
tree is shorter and wider than that in \protect\subref{Csub1btree}.
}
\label{fig:order}
\end{figure}

\begin{table}[!thb]
 \centering
\begin{tabular}{c|r|rcc|rcr}
\hline
Prob & No. & \multicolumn{3}{ c |}{$C$} & \multicolumn{3}{ c
}{$L$} \\
 & effects & nnz & $n_z$ & $\rho_C $& nnz & $n_z$ &
$\rho_L$  \\
\hline
P01 & 3488 & 56946 & 16.3 & 9.4 & 112618   & 32.3 & 18.5 \\
 P02 & 4796 & 80946 & 16.9 & 7.0 & 172023   & 35.9 & 15.0 \\
 P03 & 5892 & 105059 & 17.8 & 6.1 & 273315   & 46.4 & 15.7 \\
 P04 & 8132 & 144240 & 17.7 & 4.4 & 377761   & 46.5 & 11.4  \\
 P05 & 11711 & 209235 & 17.9 & 3.1 & 507711   & 43.4 & 7.4  \\
 P06 & 15470 & 291318 & 18.8 & 2.4 & 718701   & 46.5 & 6.0  \\
 P07 & 19146 & 370799 & 19.4 & 2.0 & 1020414   & 53.3 & 5.6  \\
 P08 & 24768 & 473891 & 19.1 & 1.5 & 1196903   & 48.3 & 3.9  \\
 P09 & 32450 & 648237 & 20.0 & 1.2 & 1779662   & 54.8 & 3.4 \\
 P10 & 44874 & 932054 & 20.8 & 0.9 & 2817463   & 62.8 & 2.8  \\
\hline
\end{tabular}
\caption{Symbolic analysis of the $LDL^T$ factorization and
selected inversion. nnz is the number of non-zero elements in
$C$ and $L$, where only the lower triangular part of $C$ is
stored. $n_z$ is the average non-zero elements per column in the
low triangular matrices.  $\rho_L$ is the ratio of
non-zero elements in the low triangular matrices (in per thousand). The FLOPs counts
for $LDL^T$ are computed according to the formula \eqref{eq:cldl}¡£}
\end{table}
\begin{table}[!thb]
 \centering
\begin{tabular}{r|r|rr}
\hline
 Prob & AMD & $LDL^T$& FLOPS count  \\
\hline
P1 & 0.0121 & 0.0093 &8943842  \\
P2 & 0.0197 & 0.0127 & 17175555  \\
P3 & 0.0266 & 0.0173 & 40768817  \\
P4 & 0.0433 & 0.0244  & 60714709 \\
P5 & 0.0633 & 0.0440 & 75897428 \\
P6 & 0.0845 & 0.0505 &149074099  \\
P7 & 0.1177 & 0.0679 &270835518  \\
P8 & 0.1741 & 0.0854 & 290699965 \\
P9 & 0.2507 & 0.1280 &600925570 \\
P10 & 0.4010 & 0.8054&  1391099157 \\
\hline
\end{tabular}
\caption{Timing results for the benchmark problem. Column AMD is
the time for the approximate degree ordering time. Column $LDL^T$ is
the timing for the $LDL^T$ factorization.
}
\label{tab:bench}
\end{table}

\section*{Appendices}
\appendix
\newtheorem{thm}[section]{Theorem}
\begin{thm}
Let $X\in \mathbb{R}^{n\times p}$ be full rank and $P_X=X(X^TX)^{-1}X^T$,
then there exist an orthogonal matrix $K=[K_1, K_2]$, such that
\begin{enumerate}
  \item $P_X=K_1K_1^T$;
  \item $I-P_X=K_2K_2^T$.
\end{enumerate}
\label{thm:a}
\end{thm}
\begin{proof}
It is easy to verify that $P_x$ is an
symmetric \textit{projection/idempotent matrix}, i.e.
$$P_X^T=P_X,\quad P_X^2=P_X.$$ Since $P_X(I-P_X)=0$, the eigenvalues of
$P_X$ are 1 and 0. There exists an orthogonal matrix $K=(K_1,
K_2)$, $K\in \mathbb{R}^{n\times n}$, $K_1\in
\mathbb{R}^{n\times p}$, and $K_2\in \mathbb{R}^{n\times (n-p)}$
such that
\begin{equation}
P_X=(K_1, K_2) \begin{pmatrix}
I_p & 0 \\ 0 & 0
\end{pmatrix}
\begin{pmatrix}
K_1^T \\K_2^T
\end{pmatrix}=K_1K_1^T.
\end{equation}
One can show that there are exactly $p$ eigenvalues with 1.

Equivalently,
 \begin{equation}
P_X(K_1, K_2)=(K_1, K_2)\begin{pmatrix}
I_p & 0 \\ 0 & 0
\end{pmatrix}.
\end{equation}
It is clear that each column of $K_1$($K_2$) is an eigenvector of $P_X$ corresponding
to the eigenvalue 1(0). Further, one can verify that $P_XX=X$, i.e., each of the column of $X$
is an eigenvector corresponding to $1$. Since eigenvectors corresponding to different
eigenvalues are orthogonal, we have
\begin{equation}
K_2^TX=0.
\end{equation}
Further, one can verify that
\begin{align*} I&=(KK^T)(KK^T)=(K_1, K_2)\begin{pmatrix} K_1^T\\K_2^T\end{pmatrix}(K_1,K_2)\begin{pmatrix} K_1^T \\K_2^T \end{pmatrix} \\
 &=(K_1,K_2)\begin{pmatrix} K_1^T K_1  & 0 \\
 0 & K_2^TK_2 \end{pmatrix} \begin{pmatrix}  K_1^T
 \\K_2^T
 \end{pmatrix}\\
 &=K_1(K_1^TK_1)K_1^T+K_2(K_2^TK_2)K_2^T=K_1K_1^T+K_2K_2^T.
\end{align*}
We have
\begin{equation}
 I-P_X=K_2K_2^T.   \label{eq:K2K2}
 \end{equation}
  \flushright \qed
\end{proof}
\begin{thm}
Let $X\in \mathbb{R}^{n\times p}$ and $\rank{X}=p$, $p<n$. Then
there exist nonsingular matrices $L=[L_1, L_2]$, such that
$L_1^TX=I_{p\times p}$, $L_2^TX=0_{(n-p)\times p}$.
\label{thm:L}
\end{thm}
\begin{proof}
Let $B\in \mathbb{R}^{(n-p)\times (n-p)}$ is any nonsingular
matrix and $K_2K_2^T=I-P_X$ be defined in \eqref{eq:K2K2}. Then
$BK_2^TX=0$ ($K_2B^T\in \ker{X^T}$) and $\rank=K_2B^T=n-p$.
Therefore the columns of $\{X, K_2B^T\}$ forms a set of basis of
$\mathbb{R}^{n\times n}$. Denote $L^T=[X, K_2B^T]^{-1}$, then
use the identy $L^T[X,K_2B^T]=I$ , we have
\begin{equation}
\begin{pmatrix} L_1^TX  & L_1^TK_2B^T \\L_2^TX & L_2^TK_2B^T
\end{pmatrix}
=
\begin{pmatrix}
I_{p\times p} &0\\0&I_{(n-p)\times(n-p)}
\end{pmatrix}
\end{equation}
\flushright \qed
\end{proof}
\newcommand{\R}{\mathbb{R}}
\begin{thm}
Let $X\in \mathbb{R}^{n\times p}$ and $\rank{X}=p$. For any full
rank matrix $L_2\in \R^{n\times(n-p)}$, and $L_2^TX=0$, we have
\begin{equation}
I-P_X=L_2(L_2^TL_2)^{-1}L_2^T,
\end{equation}
where $P_X=X(XX)^{-1}X^T$. \label{thm:p2}
\end{thm}
\begin{proof}
Let $B=[X,L_2]$. Since the columns of B is linear independent,
therefore we have the identity
$I=BB^{-1}B^{-T}B^T=B(B^TB)^{-1}B^T$.
\begin{equation}
I=(X, L_2)\begin{pmatrix} X^TX & X^TL_2 \\
L_2^TX & L_2^TL_2
\end{pmatrix}^{-1}\begin{pmatrix} X^T \\ L_2^T
\end{pmatrix}.
\end{equation}
Use $L2^TX=0$, we have $P_X+L_2(L_2^TL_2)^{-1}L_2^T=I$.  \flushright \qed
\end{proof}

\begin{thm} Let $H\in \R^{n \times n}$ be a symmetric positive definite
matrix. $X\in \R^{n\times p}$ and $L=[L_1, L_2]\in \R^{n\times
n}L_2$ such that $L_1^TX=I_p$, $L_2^TX=0$, then
\begin{equation}
P_X^H=L_2(L_2^THL_2)^{-1}L_2^T,
\label{eq:L2HL2}
\end{equation}
where $P_X^H=H^{-1}-H^{-1}X(X^TH^{-1}X)^{-1}X^TH^{-1}$,
and
 \begin{equation}
 (X^TH^{-1}X)^{-1}=L_1^THL_1-L_1^{T}HL_2(L_2^THL_2)^{-1}L_2^THL_1.
 \label{eq:XHX}
 \end{equation}
 \label{thm:d}
\end{thm}
\begin{proof}
Since $H$ is symmetric positive definite, then there exist a
symmetric positive definite $H^{1/2}$.% such that$H=H^{1/2}H^{1/2}$.
Let $\hat{X}=H^{-1/2}X$, then for  $\hat{X}\in
\R^{n\times p}$,  $L_2^TH^{1/2} \hat{X}=0$.  According to Theorem
\ref{thm:p2}, we have
\begin{equation}
I-P_{\hat{X}}=H^{1/2}L_2(L_2^THL_2)^{-1}L_2^TH^{1/2}.
\end{equation}
Multiply $H^{-1/2}$ on left and right on both side of the
equation, we obtain
\begin{equation}
H^{-1}-H^{-1}X(X^TH^{-1}X)^{-1}X^TH^{-1} = L_2(L_2^TH^{-1}L_2)^{-1}L_2.
\end{equation}

Using the equation \eqref{eq:L2HL2} on the right hand side of
\eqref{eq:XHX},
we have
\begin{align*}
   &L_1^THL_1- L_1^TH\underbrace{L_2(L_2^THL_2)^{-1}L_2^T}_{=P_X^H}HL_1\\
  =&L_1^THL_1-L_1^T\underbrace{(H-X(X^TH^{-1}X)^{-1}X^T)}_{=HP_X^HH}L_1 \\
  = &
  \underbrace{L_1^TX}_{=I_p}(X^TH^{-1}X)^{-1}\underbrace{X^TL_1}_{=I_p}=(X^TH^{-1}X)^{-1}
\end{align*}
\flushright \qed
\end{proof}

Apply the result in
\eqref{eq:L2HL2} and \eqref{eq:XHX}, we can remove the term
$L_2$ in \eqref{eq:l2},

\begin{thm}
The matrix $P=H^{-1}-H^{-1}X(X^TH^{-1}X)^{-1}X^TH^{-1}$ can also be written
as $R^{-1}-R^{-1}WC^{-1}W^TR^{-1}$,  where $W=(X ,Z)$
\end{thm}
\begin{acknowledgements}
The author would like thank Prof Robin Thompson at Rothamsted Research and Dr Sue Welham at VSN international Ltd for introducing the linear mixed model and the AI-REML method.
\end{acknowledgements}

\bibliographystyle{spmpsci}      % mathematics and physical sciences
%\bibliography{Reportv2}

\end{document}